\documentclass[12pt,a4paper]{amsart}
\usepackage{amsmath,amsthm,amsfonts,amssymb,latexsym,mathrsfs}

\usepackage{url}

\usepackage[left=3cm,right=4cm,top=1.5cm,bottom=1.5cm]{geometry} % set text area
\usepackage{enumerate}

\newtheorem{theorem}{Theorem}[section]
\newtheorem{proposition}[theorem]{Proposition}
\newtheorem{lemma}[theorem]{Lemma}
\newtheorem{claim}[theorem]{Claim}

\newtheorem{corollary}[theorem]{Corollary}
\newtheorem{observation}[theorem]{Observation}
\theoremstyle{definition}

\newtheorem{question}[theorem]{Question}

\newcommand{\U}{\mathcal U}
\newcommand{\w}{\omega}

\newcommand{\IQ}{\mathbb Q}

\newcommand{\IP}{\mathbb P}

\newcommand{\B}{\mathcal{B}}
\newcommand{\C}{\mathcal{C}}

\newcommand{\A}{\mathcal{A}}
\newcommand{\G}{\mathcal{G}}
\newcommand{\I}{\mathcal{I}}

\newcommand{\X}{\mathcal{X}}
\newcommand{\Y}{\mathcal{Y}}

\newcommand{\F}{\mathcal{F}}

\newcommand{\V}{\mathcal{V}}

\newcommand{\bb}{\mathfrak b}

\newcommand{\uhr}{\upharpoonright}

\newcommand{\name}[1]{\dot{#1}}
\newcommand{\la}{\langle}
\newcommand{\ra}{\rangle}

\newcommand{\hot}{\mathfrak}

\newcommand{\nothing}[1]{}

\let\olduparrow\uparrow
\renewcommand{\uparrow}{\mathop{\olduparrow}}

\title[Mathias forcing for filters]{Mathias forcing and
combinatorial covering properties of filters}

\author{David Chodounsk\'y, Du\v{s}an Repov\v{s}, and Lyubomyr Zdomskyy}

\address{Institute of Mathematics of the Academy of Sciences of the Czech Republic,
\v{Z}itn\'{a}~25, Praha~1, Czech Republic}
 \email{david.chodounsky@matfyz.cz}

\address{Faculty of Education, and Faculty of Mathematics and Physics,
University of Ljubljana, P. O. Box 2964, Ljubljana, Slovenia 1001.}
\email{dusan.repovs@guest.arnes.si}
\urladdr{http://www.fmf.uni-lj.si/\~{}repovs/index.htm}

\address{Kurt G\"odel Research Center for Mathematical Logic,
University of Vienna, W\"ahringer Stra\ss e 25, A-1090 Wien,
Austria.}
\email{lzdomsky@gmail.com}
\urladdr{http://www.logic.univie.ac.at/\~{}lzdomsky/}


\thanks{The first author was partially supported by the grant IAA100190902 of GA AV \v{C}R and  RVO: 67985840.
The second author was supported by the Slovenian Research Agency grants
P1-0292-0101 and J1-5435-0101.
The third author would
like to thank the Austrian Academy of Sciences (APART Program) as
well as the Austrian Science Fund FWF (Grant I 1209-N25)
 for generous support for this research.
 The collaboration of the first and the third author was partially supported by the
Czech Ministry of Education grant 7AMB13AT011: Combinatorics and Forcing.}

\begin{document}
\begin{abstract}
We give  topological characterizations of filters $\F$
on $\w$ such that the Mathias forcing $\mathbb M_\F$ adds no
dominating  reals or preserves ground model unbounded families. This
allows us to answer some questions of Brendle, Guzm\'an,
Hru\v{s}\'ak, Mart\'{\i}nez,  Minami, and Tsaban.
\end{abstract}

\maketitle

\section{Introduction}

A subset $\F$ of $[\w]^{\w}$ is called a \emph{filter} if $\F$
contains all co-finite sets, is closed under
finite intersections of its elements, and under taking supersets.
Every filter $\F$ gives rise  to a natural forcing notion $\mathbb M_\F$
introducing a generic subset $X\in [\w]^\w$ such that $X\subset^* F$
for all $F\in\F$ as follows: $\mathbb M_\F$ consists of pairs
$\la s,F\ra$ such that $s\in [\w]^{<\w}$, $F\in\F$, and $\max s<\min F$.
A condition $\la s,F\ra$ is stronger than $\la t,G\ra$
if $F\subset G$, $s$ is an end-extension of $t$, and
$s\setminus t\subset G$. $\mathbb M_\F$ is usually called \emph{Mathias forcing
associated with} $\F$.

Posets of the form $\mathbb M_\F$
are important in the set theory of reals and  have been
used to establish various consistency results, see, e.g., \cite{Can88, HruMin??}
and references therein. One of the most fundamental questions about
$\mathbb M_\F$ is whether it adds a dominating real, i.e.,
whether in $\w^\w$ of the generic extension $V^{\mathbb M_\F}$ there exists $x$
such that for every $a\in \w^\w $ in the ground model $V$ the inequality $a(n)\leq x(n)$
holds for all but finitely many $n$. Such filters $\F$ admit the following
 topological characterization proved in section 2.

\begin{theorem} \label{main}
Let $\F$ be a filter. Then $\mathbb M_\F$ does not add dominating reals
 if and only if $\F$ has the  Menger covering property as a subspace of $\mathcal P(\w)$.
\end{theorem}

Recall from \cite{COC2} that a topological space $X$ has the
\emph{Menger covering property} (or simply is Menger), if for every sequence $\la \U_n\colon
n\in\omega\ra$ of open covers of $X$ there exists
a sequence $\la \V_n\colon  n\in\omega\ra$ such that $\V_n\in [\U_n]^{<\w}$ and
$\left\{\bigcup\V_n\colon  n\in\omega\right\}$ is a cover  of  $X$.
Menger spaces can be equivalently characterized as spaces $X$ such that
no image of $X$ via a continuous function from $X$ to $\omega^\omega$ is $<^*$-dominating.

Theorem~\ref{main} has a number of applications. For instance,
since   analytic Menger sets of reals are
$\sigma$-compact  \cite{Ark86},
it implies the following
fact\footnote{While completing this
manuscript we have learned that Corollary~\ref{answer1} for Borel filters
 has been independently
obtained in \cite{GuzHruMar??}.} answering \cite[Question~4.3]{HruMin??} in the negative.

\begin{corollary}\label{answer1}
Let $\F$ be an analytic filter on $\w$. Then $\mathbb M_{\F}$ does
not add a dominating real if and only if $\F$ is $\sigma$-compact.
\end{corollary}

Several additional  applications of Theorem~\ref{main} will be
presented in Section~\ref{applications_menger}.

Following \cite{GerNag82}  we say that a family  $\U$ of subsets of
a  set $X$ is
\begin{itemize}
\item an \emph{$\omega$-cover}, if $X\not\in \U$
and for every finite subset $K$ of $X$ there exists $U\in\U$ such
that $K\subset U$;
\item a \emph{$\gamma$-cover}, if for every $x\in X$ the family
$\{U\in \U\colon x\not\in U\}$ is finite.
\end{itemize}
The Hurewicz (resp. Scheepers\footnote{In \cite{COC1} this property is denoted by
$U_{\mathit{fin}}(\mathcal O,\Omega)$. The name ``Scheepers
property'' was suggested by Banakh and by now seems to have become quite
standard.}) property is defined in the same way as the Menger one,
the only difference being that  the family $\left\{\bigcup\V_n\colon n\in\w\right\}$ must be a
$\gamma$-cover (resp. $\w$-cover) of $X$.

We say that a poset $\IP$ is \emph{almost $\w^\w$-bounding}
if for every $\IP$-name $\dot f$ for a real and $q \in \IP$,
there exists $g\colon \omega \to \omega$ such that for every
$A \in {[\omega]}^\omega$ there is $q_A \leq q$ such that
$q_A \Vdash g \restriction A \not<^* \dot{f}\restriction A$.
It is well known that almost $\w^\w$-bounding posets preserve
unbounded families of reals of the ground model as unbounded
families in the generic extension. This was observed by
Shelah in~\cite{shelah-alm_bound}. The following lemma shows
that this property in fact characterizes almost $\w^\w$-bounding
posets. We are not aware of this fact having been mentioned in the literature before.

% if for every unbounded $X\subset\w^\w$, $X\in V$, we have that
% $1\Vdash_\IP$ ``$X$ is unbounded''.
% This definition is in fact equivalent to the usual one, for one direction see e.g.\,Lemma~6 in~\cite{GuzHruMar??},
% the other direction is proved as the following lemma.

% \begin{lemma}
%   Suppose there is a name $\dot f$ for a real and $p \in \IP$ a condition such that for all
% $g \in \w^\w$ there is an infinite set
%   $A_g \in [\w]^\w$ such that $p \Vdash g\restriction A_g \leq^* \dot f\restriction A_g$.
%   Then there is an~unbounded $X\subset\w^\w$, such that $p\Vdash_\IP$ ``$X$ is bounded by $\dot f$''.
% \end{lemma}

\begin{lemma}
    A poset $\mathbb P$ is almost $\w^\w$-bounding if and only if $\mathbb P$
    preserves all unbounded families of the ground model as unbounded families
    in the extension.
\end{lemma}
\begin{proof}
    Suppose that $\mathbb P$ is not almost $\w^\w$-bounding.
    There is a name $\dot f$ for a real and $p \in \IP$, a condition such that for all
    $g \in \w^\w$ there is an infinite set
    $A_g \in [\w]^\w$ such that $p \Vdash g\restriction A_g \leq^* \dot f\restriction A_g$.
    For every $g \in \w^\w$ define $g'(n) = g(n)$ if $n \in A_g$, and $g'(n) = 0$ otherwise.
    The set $X = \left\{g'\colon g \in \w^\w\right\}$ is an unbounded set of reals,
    and the condition $p$ forces  $X$ to be bounded by $\dot{f}$ in the extension.

    Suppose that $\mathbb P$ is almost $\w^\w$-bounding and
    let $X$ be an unbounded set of reals.
    Let $\dot f$ be a name for a real and $q \in \mathbb P$  a condition.
    Find $g\colon \omega \to \omega$ as in the definition of an almost $\w^\w$-bounding
    forcing. Since $X$ is unbounded, there is $h \in X$ such that
    $A = \{n \in \omega : g(n) < h(n) \}$ is infinite.
    Now $q_A \leq q$ forces $h$ to be not dominated by $\dot{f}$.
\end{proof}

The following theorem is the main result of section~\ref{444}.

\begin{theorem} \label{main2}
Let $\F$ be a filter. Then $\mathbb M_\F$ is almost $\w^\w$-bounding %preserves ground model unbounded families
 if and only if $\F$ has the Hurewicz property.
\end{theorem}

Theorem~\ref{main2} turns out to have applications  to general
Hurewicz spaces, not only to filters. In order to formulate them we
need to recall some definitions. A Tychonov space $X$ is called a
\emph{$\gamma$-space} \cite{GerNag82} if every open $\w$-cover of
$X$ contains a $\gamma$-subcover. $\gamma$-spaces are important in
the theory of function spaces as they are exactly those $X$ for which
 the space  $C_p(X)$  of continuous functions from $X$ to
$\mathbb R$, with the topology inherited from $\mathbb R^X$, has the
Fr\'echet-Urysohn property.

 For $a\in [\w]^\w$  and $n\in\w$, $a(n)$ denotes the $n$-th
element in the increasing enumeration of $a$. For $a, b \in
[\w]^\w$, $a\leq^*b$
 means that $a(n) \leq b(n)$
for all but finitely many $n$. A $\mathfrak b$-scale is an unbounded
set $S = \{s_\alpha\colon \alpha<\mathfrak b\}$ in $([\w]^\w,\leq^*)$ such that
$s_\alpha\leq^*s_\beta$ for $\alpha<\beta$. It is easy to see that
$\mathfrak b$-scales exist in  ZFC. For each $\mathfrak b$-scale $S$,
$S\cup [\w]^{<\w}$ is $\mathfrak b$-concentrated on $[\w]^{<\w}$ in
the sense that $|S\setminus U|<\mathfrak b$ for any open $U\supset
[\w]^{<\w}$. For brevity, the union of a $\mathfrak b$-scale with
$[\w]^{<\w}$,
 viewed as a subset of the Cantor space $\mathcal P(\w)$,
will be called \emph{a $\mathfrak b$-scale set}. As an application of
Theorem~\ref{main2} we will get the following result answering
\cite[Problem~4.2]{RepTsaZdo08} in the affirmative.

\begin{corollary} \label{hur_imp_gamma}
It is consistent with ZFC that every $\hot b$-scale set is a
$\gamma$-space.
\end{corollary}

The study of the relation between $\hot b$-scale sets and
$\gamma$-spaces already has some history. First of all, in the Laver model all
$\gamma$-subspaces of $2^\w$ are countable because they have strong measure zero
\cite{GerNag82}. Answering one of the questions posed in
\cite{GerNag82},   Galvin and Miller \cite{GalMil84} constructed
under $\hot p=\hot c$ a $\hot b$-scale set which is a $\gamma$-set.
Their  $\mathfrak b$-scale was a tower, where   $S =
\{s_\alpha\colon \alpha<\kappa\}\subset [\w]^\w$ is called a \emph{tower}
if $s_\alpha\subset^* s_\beta$ for all $\beta<\alpha$ and $S$ has no pseudointersection.
Later Orenshtein and Tsaban proved \cite{OreTsa11} that if $\hot p=\hot b$
then any $\hot b$-scale set is  a $\gamma$-space provided that the
corresponding $\hot b$-scale is a tower. On the other hand, under
$\hot b=\hot c$ there exists  a $\hot b$-scale set
which fails to be a $\gamma$-space, see \cite{RepTsaZdo08}.
Also, it is easy to show that such
$\hot b$-scale sets exist under $\hot p<\hot b$, see Observation~\ref{p_less_b}.
Thus  $\hot p=\hot b<\hot
c$ holds in any model
of Corollary~\ref{hur_imp_gamma}.

In section~\ref{good_applications} we discuss  when
an unbounded subset of $\w^\w$ can be made bounded by forcing without introducing
dominating reals. Some partial answers are give for filters, see
Theorem~\ref{turn_into_hur} and the Remark 5.6 at the end of the section.

While dealing with the covering properties of Menger, Hurewicz, and Scheepers,
as well as that of being the $\gamma$-space, we shall freely use that they are
(as almost all natural covering properties)  inherited by continuous images
and closed subspaces. In addition, the properties of Menger and Hurewicz are
preserved by products with $\sigma$-compact spaces and by countable unions.
These straightforward facts  exist 
% do st.4!!!!!!!!!
 in the literature, but
we do not give any references because we believe that the reader
will need just a couple of minutes to check any of them.

For the definitions of cardinal characteristics  used in
this paper we refer the reader to  \cite{Vau90}.

\section{Proof of Theorem~\ref{main}}

Subsets of $\mathcal P(\w)$ are considered as usual
with the topology inherited from $\mathcal P(\w)$, which
is identified with the Cantor space $2^\w$ via characteristic functions.
For every $n\in\w$ and $q\subset n$ we denote  the set
$\{A\in\mathcal P(\w)\colon A\cap n=q\}$ by $[n,q]$.
The sets of the form $[n,q]$ form a base for the standard 
topology of $\mathcal P(\w)$. Set also $\uparrow X=\{A\in\mathcal P(\w)\colon A \supset X\}$
for every $X\subset\w$.

\begin{claim} \label{clear_cover_1}
Suppose that $\X\subset\mathcal P(\w)$ is closed under taking supersets
and  $\mathcal O$ is a cover of   $\X$ by sets open in $\mathcal P(\w)$.
Then there exists a family $Q\subset [\w]^{<\w}$ such that
\[\X\subset\bigcup_{q\in Q}\uparrow q\subset \bigcup\mathcal O.\]
\end{claim}
\begin{proof}
Without loss of generality, we can assume that $\mathcal O$ consists of sets of the form
$[n,q]$.
Let us fix  $X\in \X$, notice that $\uparrow X$ is compact,
and find a finite family of basic open sets
 $\{[n_i,q_i]\colon i\in m\}\subset\mathcal O$ such that
$\uparrow X\subset \bigcup_{i\in m}[n_i,q_i]$.
Put $n = \max \{ n_i: i \in m\}$.
If $A \in \uparrow (X \cap n)$, then $(A \cap n) \cup (\omega \setminus n) \in \uparrow X$,
and there is $i \in m$ such that  $(A \cap n) \cup (\omega \setminus n) \in [n_i,q_i]$;
thus $A \in [n_i,q_i]$. We showed that
$X\in \uparrow X \subset \uparrow(X\cap n)\subset \bigcup\mathcal O$.
% Breaking some of the
% sets $[n_i,q_i]$ into smaller pieces of the same form,
% we may assume if necessary that for some $n\in\w$ we have   $n_i=n$  for all
% $i\in m$. Moreover, we can assume that no proper subcollection of
% $\{[n,q_i]\colon i<m\}$ covers $\uparrow X$. The definition of  $\uparrow
% X$  implies that $\{q_i\colon i<m\}= \{t\subset n\colon X \cap n\subset t\}$,
% and consequently $\bigcup_{i<m}[n,q_i]=\uparrow(X \cap n)$. Thus
% $X\in \uparrow X \subset \uparrow(X\cap n)\subset \bigcup\mathcal
% O$.
\end{proof}

Since every set of the form $\uparrow q$ is compact, it follows that 
for every $q\in Q$ (we use notation from Claim \ref{clear_cover_1})
there exists a finite subset $\mathcal O'\subset\mathcal O$
such that $\uparrow q\subset\bigcup \mathcal O'$.
This  gives us the following

\begin{corollary} \label{menger_for_filters}
 If $\X\subset\mathcal P(\w)$ is closed under taking supersets,
then $\X$ has the Menger property if and only if
 for every sequence $\la \U_n\colon
n\in\omega\ra$ of open covers of $\X$ by sets of the form $\uparrow
q$ for some $q\in [\w]^{<\w}$,
 there exists
a sequence $\la \V_n\colon  n\in\omega\ra$ such that $\V_n\in
[\U_n]^{<\w}$ and $\left \{\bigcup\V_n\colon  n\in\omega\right \}$ is a cover  of  $\X$.
\end{corollary}

A set $\I\subset \mathcal P(\w)$ is called \emph{an ideal},
if $\F:=\{\w\setminus I\colon I\in\I\}$ is a filter. In this case we write
$\I=\F^*$ and $\F=\I^*$.
The collection of all $\I$-positive sets $\mathcal P(\omega) \setminus \I$ is denoted by $\I^+$
or $\F^+$, if $\F$ is the filter dual to $\I$.
Following \cite{HruMin??}
we call an ideal $\I$  a \emph{$P^+$-ideal} if for every decreasing sequence $\la X_n \colon
n\in\w\ra$  of $\I$-positive sets, there is an $X\in\I^+$  such that $X \subset^* X_n$ for
all $n\in\w$. We shall also use the following notation:
\[ \I^{<\w}=\{A\subset [\w]^{<\w}\colon  \exists I\in\I\:\forall a\in A\: (a\cap I\neq\emptyset)\}. \]
It is easy to see that $\I^{<\w}$ is an ideal on $[\w]^{<\w}$,
and  letting $\F=\I^*$ we have  $\I^{<\w}=(\F^{<\w})^*$, where $\F^{<\w}$
is the filter on $[\w]^{<\w}$ consisting of sets containing $[F]^{<\w}$
for some $F\in\F$.
%The following easy claim is most likely well-known.

The following claim is known, we give here a proof for  reader's convenience.
A~stronger form of this result is presented in~\cite{LafLear}.

\begin{claim}\label{P+-characterization}
Let $\I$ be an ideal on $\w$. Then $\I$ is a $P^+$-ideal if and only if  for
every sequence $\la X_n \colon  n\in\w\ra$  of $\I$-positive sets there is
a sequence $\la Y_n \colon  n\in\w\ra$ of finite sets such that
$Y_n\subset X_n$ and $\bigcup_{n\in\w}Y_n\in\I^+$.
\end{claim}
\begin{proof}
The ``if'' part is obvious. To prove the ``only if'' part fix a
sequence $\la X_n \colon  n\in\w\ra$  of $\I$-positive sets
and set $X'_n=\bigcup_{m\geq n}X_m$ for all $n\in\w$. Then
$\la X'_n \colon  n\in\w\ra$ is a decreasing sequence of  $\I$-positive sets,
and hence there exists $Y\in\I^+$ such that $Y\subset^* X'_n$ for all
$n\in\w$. Without loss of generality we may assume that $Y\subset X'_0$.
For every $y\in Y\setminus\bigcap_{n\in\w}X'_n$
let $n(y)$ be the maximal $n$ such that $y\in X'_n$.
For $y\in Y\cap\bigcap_{n\in\w}X'_n$  let $n(y)$ be any $n>y$.
% such that $y\in X_n$.
Then $Y_n=\{y\in Y\colon  n(y)=n \}$ is finite, $Y_n\subset X_n$,
and $Y=\bigcup_{n\in\w}Y_n$.
\end{proof}

The proof of the following fact is more or less just
a reformulation.

\begin{claim}\label{trivial_but_funny}
Let $\I$ be an ideal. Then $\I^{<\w}$ is a $P^+$-ideal
if and only if $\I$ is a Menger  subspace of $\mathcal P(\w)$.
\end{claim}
\begin{proof}
Since $\I$ is homeomorphic to $\I^*$ it is enough to show that $\I^*$
is Menger if and only if $\I^{<\w}$ is a $P^+$-ideal.

Assume that $\I^{<\w}$ is a $P^+$-ideal and fix a sequence $\la
\U_n\colon n\in\omega\ra$ of open covers of $\I^*$ by sets of the form
$\uparrow a$ for some $a\in [\w]^{<\w}$. Set $A_n=\{a\colon \uparrow
a\in\U_n\}$. Since $\U_n$ covers $\I^*$, for every $F\in\I^*$ there
exists $a\in A_n$ such that $a\subset F$, which means that $A_n$ is
$\I^{<\w}$-positive. Therefore there exists a sequence $\la
B_n\colon n\in\w\ra$ such that $B_n\in [A_n]^{<\w}$ and
$B=\bigcup_{n\in\w}B_n\in \left(\I^{<\w}\right)^+$. This means that for
every $F\in \I^*$ there exists $b\in B$ such that $b\subset F$,
i.e., that $\{\uparrow b\colon b\in B\}$ covers $\I^*$. Thus for every $n$
we can select a finite subset of $\U_n$ (namely $\V_n = \{\uparrow
b\colon b\in B_n\}$) whose union covers $\I^*$. By
Corollary~\ref{menger_for_filters} this means that $\I^*$ is Menger.

Now suppose that $\I^*$ is Menger and fix a sequence
$\la A_n\colon n\in\w\ra$ of $\I^{<\w}$-positive sets.
For every $n$ set $\U_n=\{\uparrow a \colon  a\in A_n\}$ and notice that
$\U_n$ is a cover of $\I^*$ by sets open in $\mathcal P(\w)$.
Thus for every $n$ there exists a finite $\V_n\subset\U_n$
such that $\bigcup_{n\in\w}\V_n\supset\I^*$. Let $B_n\in [A_n]^{<\w}$
be such that $\V_n=\{\uparrow a\colon a\in B_n\}$.
It follows that for every $F\in\I^*$ there exists
$a\in\bigcup_{n\in\w}B_n$ such that $a\subset F$. In other words,
$\bigcup_{n\in\w}B_n$ is  $\I^{<\w}$-positive, which completes our proof.
\end{proof}

Now Theorem~\ref{main} is a direct consequence of Claim~\ref{trivial_but_funny},
the fact that $\F$ is homeomorphic to $\F^*$ for any filter $\F$,
and  the following important

\begin{theorem}\textup{\cite[Theorem 3.8]{HruMin??}}
 Let $\I$ be an ideal on $\w$. Then $\mathbb M_{\I^*}$ does not add a
dominating real if and only if $\I^{<\w}$ is a $P^+$-ideal.
\end{theorem}

\section{Straightforward applications of Theorem~\ref{main}} \label{applications_menger}

Recall that $A,B $ are called \emph{almost disjoint}, if
$A\cap B$ is finite. Given a countable set $I$, an infinite set $\A\subset [I]^\w$ is said to be an
\emph{almost disjoint family} (on $I$)  if
any two elements of $\A$ are almost disjoint. $\A$ is called a
\emph{mad family} (on $I$), if it is
 maximal with respect    to inclusion among
almost disjoint families on $I$. 
Every almost disjoint family $\A$ generates an ideal
\[ \I(\A) = \left\{I \subset \w \colon \exists \B \in [\A]^{<\w} \: \left(I \subset^* \bigcup \B \right)\right\}.  \]
%denoted by $\I(\A)$.
The dual filter is denoted by $\F(\A)$.
% It is easy to verify that
% its dual ideal (denoted by $\I(\A)$ in what follows)
%  consists of sets $X$ such that $X\subset^*\cup\B$
% for some  $\B\in [\A]^{<\w}$.
Theorem~\ref{main}  allows us to give an easy proof of the
following recent result of Guzm\'an, Hru\v{s}\'ak and Mart\'{\i}nez \cite[Proposition~6]{GuzHruMar??}, answering
\cite[Question~2.7]{Bre98} in the negative.

\begin{proposition} \label{answer2}
There exists a mad family $\A$ on $\w$ such that $\mathbb M_{\F(\A)}$
adds a dominating real.
\end{proposition}
\begin{proof}
By Theorem~\ref{main} it is enough to construct a
mad family   $\A$ on $2^{<\w}$
such that $\I(\A)$ is not Menger.
Set $\C=\{C_x \colon x \in 2^\w\}$, where  $C_x=\{x\uhr n \colon  n \in \w\}$.
Then $\C$ is a compact almost disjoint family. Take a dense countable subset
$\C'$ of $\C$ and  for every $C$ in $\C'$ fix an infinite mad family
$\A_C$ of infinite subsets of $C$. Consider
$\A_0=\left(\C\setminus \C' \right) \cup  \bigcup_{C \in \C'}\A_C$
and extend $\A_0$ to a mad family $\A$ of infinite subsets of $2^{<\w}$.

We claim that $C\setminus\bigcup \B$ is infinite for all $\B\in [\A]^{<\w}$ and $C\in \C'$.
Indeed, let us fix $C, \B$, and  $A\in\A_C \setminus \B$.
Then all elements of
$\B$ have finite intersection with $A$ and hence $A\not\subset^*\bigcup\B$.
Therefore $C\not\subset^*\bigcup\B$ as well.

Thus $\I(\A)\cap \C=\C\setminus \C'$,
and hence $\I(\A)$ contains a closed copy of $\w^\w$. It remains to
note that $\w^\w$ is not Menger and that the Menger property is inherited
by closed subsets.
\end{proof}

% It is well-known and easy to show that
As mentioned before, 
the Menger and Hurewicz properties are preserved by closed subspaces
and products with compact spaces,
continuous images, and countable unions.
Thus if a filter $\F$ on $\w$ has a
base $\B$ which is  Menger (Hurewicz),
then $\F$ is Menger (Hurewicz) as well:
$\F=\psi[\B\times\mathcal P(\w)]$, where $\psi(B,X)=B\cup X$,
and $\psi$ is continuous.

Let $\U$ be an ultrafilter. For  $x,y\in\w^\w$ the notation $x\leq_\U
y $ means that $\{n \colon  x(n)\leq y(n)\}\in\U$.
We will also use the notation $A \leq_\U B$ for $A, B \in {[\omega]}^\omega$
by interpreting sets as their enumerating functions.
The relation $\leq_\U$
is a linear pre-ordering of $\w^\w$ whose cofinality is usually
denoted by $\hot d(\U)$.

Another application of Theorem~\ref{main} is the following result
improving \cite[Proposition~8]{GuzHruMar??} and partially answering \cite[Problem~2]{GuzHruMar??}.
 Instead of proving it directly
 we shall give a more streamlined argument  using
\cite[Proposition~8]{GuzHruMar??}.

\begin{proposition} \label{d_eq_c}
If \ $\hot d=\hot c$, then there exists an infinite mad family $\A$ such that
$\I(\A)$ is Menger.
\end{proposition}
\begin{proof}
First assume that $\hot d$ is regular and fix an enumeration
$\{S_\alpha\colon \alpha<\hot d\}$ of $[\w]^\w$ such that
$\{S_\alpha \colon \alpha < \w\}$ is an almost disjoint family. It is well-known
\cite{Can89} that there exists an ultrafilter $\U$ with $\hot
d(\U)=\mathrm{cf}(\hot d)$, which equals $\hot c$ in our case.
 We shall construct
$\A=\{A_\alpha\colon \alpha<\hot d\}$ by induction. At stage  $\alpha$ we
 pick $A_\alpha\subset S_\alpha$ such that
$\{S_\beta\colon \beta<\alpha\}\cup\{A_\beta\colon \beta <\alpha\}\leq_\U
A_\alpha$ provided
%\footnote{Here we identify each $A\in [\w^\w]$
%with the function $a_A\in\w^\w$ such that $e_A(n)$ is the $n$th
%element of $A$ for all $n$.}
that $S_\alpha$ is almost disjoint from
$A_\beta$ for all $\beta<\alpha$. Otherwise we set $A_\alpha$ to be
equal to one of the $A_\beta$'s constructed before. This finishes
our construction of $\A$. It is clear that $\A$ is mad.
It is well known and easy to see that $\mathfrak d$ is the minimal
cardinality of non-Menger set of reals.
Hence if $|\A|<\hot d$, then all finite powers of $\A$ are Menger.
%  because
% $\hot d$ is well-known to be  the minimal cardinality of  a
% non-Menger set of reals.
If $|\A|=\hot d$ then
\cite[Cor.~4.3]{TsaZdo08} ensures that all finite powers of $\A\cup
[\w]^{<\w}$ are Menger. In any case, $\I(\A)$ is Menger because it
can be written in the form $\bigcup_{n\in\w}\I_n$, where
\[\I_n=\left\{\bigcup_{i\in n}A^i\cap X\colon  \la A^i\colon i\in n\ra\in (\A\cup [\w]^{<\w})^n, X\in\mathcal P(\w)\right\}\]
is a continuous image
 of $(\A\cup [\w]^{<\w})^n\times \mathcal P(\w)$.

Now suppose that $\hot d$ is singular.
 It has been established in
the proof of \cite[Theorem~16]{Bla86} that  $\hot u<\hot d$ yields
$\hot d(\U)=\hot d $ for any ultrafilter $\U$ generated by $\hot u$ many
sets. Thus $\hot u<\hot d$ implies that $\hot d$ is regular, and
hence we have $\hot u=\hot d=\hot c$ by the singularity of $\hot d$.
Now $\min \{\hot d, \hot u\} \leq \hot r$ (see~\cite{Aubrey}) implies
$\hot r =\hot c$ and
it suffices to apply \cite[Prop. 8]{GuzHruMar??} which states that
under $\hot d=\hot r=\hot c$ there exists a mad family generating
a~Menger ideal.
\end{proof}

\begin{observation} \label{triv12}
If a filter $\F$ is Menger (Hurewicz),
then so is $\F^{<\w}$.
\end{observation}

Note that the converse implication is also true since
$\F$ is isomorphic to a closed subset of $\F^{<\w}$.

\begin{proof}
The map $\phi\colon \F\to \mathcal P([\w]^{<\w})$ assigning
to $F\in\F$ the set $[F]^{<\w}$ is continuous.
Thus $\F^{<\w}$ has a Menger (Hurewicz) base and hence is Menger (Hurewicz).
\end{proof}

Combining Theorem~\ref{main}, Observation~\ref{triv12}, and  \cite[Prop.\,5]{GuzHruMar??}
we get a negative answer to \cite[Problem 4]{GuzHruMar??}\footnote{The
formulation of \cite[Problem 4]{GuzHruMar??} involves notions
which will not be used in our paper, and hence we refer
the reader to \cite{GuzHruMar??} for its precise formulation.}.
\smallskip

Following \cite{GuzHruMar??} (Laflamme \cite{Laf89} for ultrafilters)
we say that a filter $\F$ is a \emph{strong $P^+$-filter} if for every
sequence $\la \C_n\colon n\in\w\ra$ of compact subsets of $\F^+$
there exists an increasing sequence $\la k_n\colon n\in\w\ra$
of integers such that if $X_n\in\C_n$ for all $n$,
then $\bigcup_{n\in\w}\left(X_n\cap \left[k_n,k_{n+1}\right)\right)\in\F^+$.
The characterization of $P^+$-filters given in Claim~\ref{P+-characterization}
implies that every strong $P^+$-filter is a $P^+$-filter.

We shall need the following game of length $\omega$ on a topological space $X$:
 In the $n$th  move player $I$ chooses a countable open cover
$\U_n$ of $X$, and player $\mathit{II}$ responds by choosing a
finite $\V_n\subset \U_n$. Player $\mathit{II}$ wins the game if $\bigcup_{n\in\omega}
\bigcup\V_n =X$. Otherwise, player $I$ wins.  We shall call this game
 \emph{the Menger game}\footnote{In case $X$ is a filter on $\omega$, notice the similarity of this game
to the game $\mathfrak G ({\F}^+,[\omega]^{<\omega},{\F}^+)$ from~\cite{LafLear}, 
where $\F=X^{<\omega}$.} on $X$.
 It is well-known  that
  $X$  is Menger  if and only if
 player $I$ has no winning strategy in the Menger game on
$X$, see \cite{Hur25} or \cite[Theorem~13]{COC1}.
Note that if $I$ plays with covers $\U_n$ closed under finite unions,
then we can assume that the player $\mathit{II}$ replies by choosing
one-element subsets of the $\U_n$'s.

The following result together with \cite[Proposition~3]{GuzHruMar??}
answers \cite[Problem~3]{GuzHruMar??} in the negative.

\begin{proposition} \label{stron_f_plus}
Every Menger filter $\F$ is a strong $P^+$-filter.
\end{proposition}
\begin{proof}
Let $\la \C_n\colon n\in\w\ra$  be a sequence of compact subsets of $\F^+$, and
assume without loss of generality that $\C_n \subseteq \C_m$ for $n<m$.
For every $F \in \F$ consider an increasing sequence $\la k_{n}^F\colon n\in\w\ra$
defined as follows: $k_{0}^F=0$, and $k^F_{n+1}$ is the minimal integer
such that $[k^F_{n},k^F_{n+1})\cap F\cap X\neq\emptyset$ for all $X\in\C_n$.
The existence of such a number follows by the
 compactness of $\C_n$. Moreover, it is easy to see that
 $F\mapsto \la k^F_{n}\colon n\in\w\ra$ is a continuous   map
from $\F$ to $\w^\w$, and hence its range $K:=\{\la k^F_{n}\ra_{n\in\w}\colon F\in\F\}\subset\w^\w$
is Menger.

Let us consider the following strategy of the  player $I$
in the Menger game on~$K$: $I$ starts by choosing
the cover $\U_0=\{U^0_m\colon m\in\w\}$ of $K$, where $U^0_m$ is the set of all
$\la k^F_{n}\ra_{n\in\w}\in K$ such that $k^F_{1}<m$.
Suppose that $\mathit{II}$ replies by choosing $U^0_{k_0}$.
Then in the next move  $I$ chooses a cover
$\U_1=\{U^1_m\colon m\in\w, m > k_0\}$ of $K$, where $U^1_m$ is the set of all
$\la k^F_{n}\ra_{n\in\w}\in K$ such that $k^F_{k_0+1}<m$.
If  $\mathit{II}$ replies by choosing $U^1_{k_1}$, then
 $I$ chooses
$\U_2=\{U^2_m\colon m\in\w, m> k_1\}$, where $U^2_m$ is the set of all
$\la k^F_{n}\ra_{n\in\w}\in K$ such that $k^F_{k_1+1}<m$, and so on.
Since $K$ is Menger, the strategy of $I$ defined above is not winning.
Therefore there exists a play $\la \U_n, U^n_{k_n}\colon n\in\w \ra$
in which $I$ follows this strategy and looses, i.e., $K\subset\bigcup_{n\in\w}U^n_{k_n}$.
% Let $\la k_n\colon n\in\w\ra$ be any increasing sequence such that
% $k_{2m_n}=3m_n$ and
%  $k_{2m_n+1}=m_{n+1}$
% for any $n$.
We claim that the sequence $\la k_n\colon n\in\w\ra$ is as required.
For this we shall show that for any $F\in\F$ and any sequence
$\la X_n \in \C_n \colon n\in\w\ra$
%if $X_n\in \C_{2m_n}$ for all $n$, then
there exists $n$ such that $X_n \cap F \cap [k_{n},k_{n+1})\neq\emptyset$.
Indeed, since $K\subset\bigcup_{n\in\w}U^n_{k_n}$, it follows that there exists
$n$ such that $k_n \leq k^F_{k_n} < k^F_{k_n+1}<k_{n+1}$.
% Also, since
% $k_{l+1,F}\geq k_{l,F}+4$ for all $l$, we conclude that
% $k_{2m_n,F}\geq 4(2m_n-1)\geq 3m_n=k_{2m_n}$.
% Thus $[k_{2m_n,F},k_{2m_n+1,F})\subset [k_{2m_n},k_{2m_n+1})$.
% By the definition of the sequence
% $\la k_{l,F}\ra_{l\in\w}$ we have that $X_n\cap F\cap [k_{2m_n,F},k_{2m_n+1,F})\neq\emptyset,$
% and hence $X_n\cap F\cap [k_{2m_n},k_{2m_n+1})\neq\emptyset$ as well,
Since $X_n \in \C_n \subseteq \C_{k_n}$,
 %and by the definition of $\la k^F_{n}\ra_{n\in\w}$
we have that $[k^F_{k_n},k^F_{k_{n}+1})\cap F\cap X_n\neq\emptyset$,
which completes our proof.
\end{proof}

To conclude this section, let us review various equivalent characterizations of filters $\F$
for which the forcing $\mathbb M_\F$ does not add dominating reals.
The following theorem combines results of this paper with results from~\cite{GuzHruMar??}.

\begin{theorem}
    Let $\F$ be a filter on $\omega$. The following are equivalent:
    \begin{enumerate}
        \item $\mathbb M_\F$ does not add dominating reals,
        \item $\F$ is Menger,
        \item $\F^{<\omega}$ is Menger,
        \item $\F^{<\omega}$ is a $P^+$ filter,
        \item $\F$ is a strong $P^+$ filter,
        \item $\F^{<\omega}$ is a strong $P^+$ filter.
    \end{enumerate}
\end{theorem}

\section{Hurewicz filters and $\gamma$-spaces} \label{444}

First we shall prove Theorem~\ref{main2}.  Suppose that $\F$ is
Hurewicz, but there exists an unbounded $X\subset\w^\w$, $X\in V$,
and an $\mathbb M_\F$-name $\name{g}$ for a  function dominating $X$
(for simplicity assume that every condition forces this).
For every $x\in X$ let us find $n^x\in\w$ and a condition $\la s^x,
F^x\ra$   forcing $x(n)<\name{g}(n)$ for all $n\geq n^x$. Since $X$
cannot be covered by a countable family of bounded sets, we may
assume that $s^x$ and $n^x$ do not depend on $x$, i.e., $s^x=s_*$
and $n^x=n_*$ for all $x\in X$.

For every $m\in\w$ let $\mathcal S_m$ be the set of those $ s\in
[\w]^{<\w}$ such that $\max s_*<\min s$  and there exist $F_s\in\F$
such that $\la  s_*\cup s , F_s \ra$ forces $\name{g}(m)$ to be
equal to some $g_{s}(m)$. It is clear that for every $F\in\F$ there
exists $s\in\mathcal S_m$ such that $s\subset F $. In other words, $
\U_m:=\{\uparrow s\colon   s \in \mathcal S_m\} $ is an open cover of
$\F$. Since $\F$ is Hurewicz, for every $m$ there exists a finite
$\V_m\subset\U_m$ such that $\left\{\bigcup\V_m\colon m\in\w\right\}$ is a
$\gamma$-cover of $\F$. Let $\mathcal T_m\in [\mathcal S_m]^{<\w}$
be such that $\V_m=\{\uparrow s \colon  s \in \mathcal T_m\}$ and
$f(m)=\max\{g_{s}(m)\colon s\in\mathcal T_m\}$.
We will derive a contradiction by showing that  $x <^* f$ for each $x \in X$.
Fix $x\in X$ and $l\in\w$  such that
for every $m \geq l$ there exists $s_m \in \mathcal T_m$ such that
$F^x\in  \uparrow s_m $.
Pick any  $m\geq n_*,l$.
Since $\la s_*, F^x\ra \Vdash x(m) < \dot g(m)$,
$\la s_* \cup s_m, F_{s_m}\ra \Vdash \dot g(m) \leq f(m)$,
and these two conditions are compatible, it follows that $x(m) < f(m)$.

% that $x(m)>f(m)$.  It follows that $s_m \subset F^f$, and hence $
% p:=\la s_*\cup s_m,  F^f \cap F_{s_m}\ra $ is a condition in
% $\mathbb M_\F$ stronger than both $\la s_*, F^f\ra$ and $\la s_*\cup
% s_m, F_{s_m}\ra$. However, $\la s_*\cup s_m, F_{s_m}\ra$ forces
% $\name{g}(m)=k_{m,s_m}\leq f(m)$, whereas $\la s_*, F^f\ra$ forces
% $\name{g}(m)>f(m)$, a contradiction.
\smallskip

Now suppose that $\F$ is not Hurewicz as witnessed by a sequence
$\la\U_n\colon n\in\w\ra$ of
 covers of $\F$ by sets open in $\mathcal P(\w)$. By
 Claim~\ref{clear_cover_1} we may additionally assume that
 $\U_n=\left\{\uparrow q_{m}(n)\colon m\in\w\right\}$, where $q_{m}(n)\in [\w]^{<\w}$.
 For every $F\in \F$ consider the function $x_F\in\w^\w$,
 $x_F(n)=\min\left\{m\colon F\in\uparrow q_{m}(n)\right\}$. It follows from
 the fact that $\F$ is not Hurewicz
 that $X=\left\{x_F\colon F\in\F\right\}$ is unbounded.

 Now let $G$ be the generic pseudointersection of $\F$ added by $\mathbb
 M_\F$.

\begin{claim}
    For every $n$ there exists
 $g(n)$ such that $G\setminus n\in \uparrow q_{g(n)}(n)$.
\end{claim}
\begin{proof}
The set $\U_n' = \left\{\uparrow q_{m}(n)\colon q_{m}(n) \cap n = \emptyset ,m\in\w\right\}$
 covers $\F$ because $\U_n$ is a cover of $\F$.
Hence for every $F \in \F$ there is some $\uparrow q_{m}(n) \in \U_n'$ such that
$F \in \uparrow q_{m}(n)$, and
the set of conditions $\la s, F\ra$ such that $q_{m}(n) \subseteq s \setminus n$ for some $m\in \omega$ is dense.
\end{proof}
 % By genericity, for every $n$ there exists
 % $g(n)$ such that $G\setminus n\in \uparrow q_{g(n)}(n)$ (because $\F\subset\bigcup\U_n$
 % is just another  way of stating that $\left\{q_{m}(n)\colon m\in\w\right\}\in \left(\F^{<\w}\right)^+$).
Let us fix $F\in\F$ and find $n$ such that $G\setminus n\subset F$.
Then $G\setminus n\in \uparrow q_{g(n)}(n)$ yields $F\in \uparrow
q_{g(n)}(n)$, which implies $x_F(n)\leq g(n)$. Thus $g \in \omega^\omega$
is dominating $X$, and therefore $\mathbb M_\F$ fails to preserve
ground model unbounded sets.  \hfill $\Box_{\mbox{\tiny
Theorem~\ref{main2}}}$
\medskip

\noindent\textbf{Remark 4.2.} \stepcounter{theorem} 
Theorem~\ref{main} can be proved directly using the ideas of the proof
of Theorem~\ref{main2}. On the other hand, the proof of \cite[Theorem~3.8]{HruMin??}
could be easily modified to get a combinatorial characterization of filters
$\F$ such that $\mathbb M_\F$ is almost $\w^\w$-bounding, and then Theorem~\ref{main2}
can be proved in the same way as Theorem~\ref{main}. We have deliberately presented
two approaches.
\hfill $\Box$

\medskip

By \cite[Theorem 10]{BarTsa06} every $\mathfrak b$-scale set has the
Hurewicz property  in all finite powers. Thus
Corollary~\ref{hur_imp_gamma} is a direct consequence of the
following

\begin{theorem} \label{main_tech}
It is consistent with ZFC that $\mathfrak b=\w_1$ and every Tychonov
space $X$  of size $\w_1$ is a $\gamma$-space provided
 that $X^n$ is Hurewicz for all $n\in\w$.
\end{theorem}
\begin{proof}
Using Theorem~\ref{main2}, a standard book-keeping argument taking
care of all filters $\F$ on $\w$  having a Hurewicz base $\B$ of
size $\w_1$, and the well-known fact that unbounded well-ordered by
$\leq^*$ subfamilies of $\w^\w$ are preserved at limit stages of
finite support iterations of c.c.c. posets (see, e.g.,
\cite[Lemma~6.5.7]{BarJud95}), we can perform an $\w_2$ steps finite
support iteration $\IP_{\w_2} =
\left\la\IP_\alpha,\name{\IQ}_\alpha\colon\alpha< \w_2\right\ra$
 of
c.c.c.\,posets such that in $V^{\IP_{\w_2}}$ the following holds:
\begin{enumerate}[(i)]
 \item[$(i)$] $\bb=\w_1$;
\item[$(ii)$] Every filter $\F$ on $\w$ has a pseudointersection
provided it has a Hurewicz base $\B$ of size $\w_1$.
\end{enumerate}
Here we have to use the observation that filters $\F$ as those in item
$(ii)$ above are Hurewicz (being a continuous image of
$\B\times\mathcal P(\w)$),  and the fact that if $\B$ of size $\w_1$
has the Hurewicz property in $V^{\IP_{\w_2}}$ then there exists an
$\w_1$-club $C\subset\w_2$ such that $\B\in V^{\IP_\alpha} $ and
$\B$ is Hurewicz in $V^{\IP_\alpha}$ for all $\alpha\in C$.

Now suppose that in $V^{\IP_{\w_2}}$ we have a Tychonov space  $X$
of size $\w_1$ such that all finite powers of $X$ are Hurewicz. Let
$\U$ be an $\w$-cover of $X$.  $X$ is zero-dimensional because
$|X|<2^\w$,
 and hence
passing to a refinement of $\U$, if necessary, we may assume that
$\U$ consists of clopen sets. Applying \cite[Proposition,
p.\,156]{GerNag82} we can find a countable $\V=\left\{U_n\colon
n\in\w\right\}\subset\U$ which is an $\w$-cover of $X$. Now consider
the map $\psi\colon X\to\mathcal P(\w)$, $\psi\colon x\mapsto
\left\{n\in\w\colon x\in U_n\right\}$. It follows from the above
that  $\psi$ is continuous and $\psi[X]$ is centered. Since all
finite powers of $X$ are Hurewicz, such are also all finite powers
of $\psi[X]$, and hence also all finite powers of
$\B=\left\{\bigcap\Y\colon  \Y\in\left[\psi[X]\right]^{<\w}\right\}
$ are Hurewicz as well because the latter is a countable union of
continuous images of finite powers of $\psi[X]$. Thus the filter
$\left\la \psi[X]\right\ra$ has the Hurewicz base $\B$ of size
$\left|\psi[X]\right|\leq |X|=\w_1$, and consequently  it has a
pseudointersection $J\in [\w]^\w$ by $(ii)$ above. Therefore
$J\subset^*\psi(x)$ for all $x\in X$, which means that
$\left\{U_n\colon n\in J\right\}$ is a $\gamma$-cover of $X$. This
completes the proof.
\end{proof}

Finally, we shall show that  $\hot p=\hot b<\hot
c$ holds in any model
of Corollary~\ref{hur_imp_gamma} (and hence also in those of Theorem~\ref{main_tech}).

\begin{observation} \label{p_less_b}
If $\mathfrak{p}<\mathfrak b$, then there exists a $\mathfrak b$-scale set
which is not a $\gamma$-space.
\end{observation}
\begin{proof}
It is easy to see that any centered subset $X$ of $[\w]^\w$
without a pseudointersection is not a $\gamma$-space: consider the open
$\w$-cover $\{O_n:n\in\w\}$ of $X$, where $O_n=\{x:n\in x\}$.
Thus there exists $X\subset [\w]^\w$
of size $\mathfrak p$
which is not a $\gamma$-space (this fact has been attributed to \cite{GalMil84} in \cite{COC2}).
Now let $\{s_\alpha:\alpha<\mathfrak b\}$ be a $\mathfrak b$-scale such that
$3$ divides $s_\alpha(n)$ for all $\alpha,n$. Since $s_\alpha<^*s_{\mathfrak p}$
for all $\alpha<\mathfrak p$ and $\mathfrak p$ is regular, there exists $n$ such that the set
$I_n=\{\alpha:s_\alpha(m)<s_{\mathfrak p}(m)$ for all $m\geq n\}$ has size $\mathfrak p$.
Without loss of generality we may assume that $n=0$, otherwise just redefine
$s_\alpha(k)$ to be equal to $3k$ for all $\alpha\in I_n$ and $k<n$ and note that
resulting functions still form a $\mathfrak b$-scale. Also, we can additionally assume
that $I_0=\mathfrak p$ because otherwise we can set $s'_\alpha=s_{\xi_\alpha}$,
where $\xi_\alpha$ is the $\alpha$th element of $I_0$, and consider the
$\mathfrak b$-scale $\{s'_\alpha:\alpha<\mathfrak p\}\cup\{s_\alpha:\alpha\geq\mathfrak p\}$.

Let us write  $X$ in the form $\{x_\alpha:\alpha<\mathfrak p\}$
and set $t_\alpha(n)$ to be the  even (resp. odd) element in the set $\{s_\alpha(n),s_\alpha(n)+1\}$
if $n\in x_\alpha$ (resp. if $n\not\in x_\alpha$). It follows that $t_\alpha \leq^* s_\beta $
for all $\alpha<\mathfrak p$ and $\beta\geq \mathfrak p$, and
$t_\alpha \leq^* t_\beta $ for all $\alpha<\beta<\mathfrak p$. Thus
$T:=\{t_\alpha:\alpha<\mathfrak p\}\cup\{s_\alpha:\alpha\geq\mathfrak p\}$
is a $\mathfrak b$-scale. Moreover, $T_0:=\{t_\alpha:\alpha<\mathfrak p\}$
is a closed subset of $T\cup [\w]^{<\w}$ because
$T_0= T\cap K $ for the compact subset $K=\{a\in [\w]^\w:a(n)\leq s_{\mathfrak p}(n)$ for all $n\}$
of $\mathcal P(\w)$. Since $T_0$ can be continuously mapped onto $X$
(using the parity), it is not a $\gamma$-space, and hence $T\cup [\w]^{<\w}$
also fails to be a $\gamma$-space because this property is inherited by closed subspaces.
\end{proof}

\section{Turning sets of reals into Hurewicz spaces
 without adding dominating reals} \label{good_applications}

It has been proven in \cite{SchTal10} that after adding $\w_1$-many
Cohen reals any set of ground model reals becomes Menger. The same
argument proves that after iterating with finite supports
c.c.c.\,posets adding dominating reals uncountably many steps, each set of
ground model reals becomes Hurewicz. The natural question which
arises here is which sets of reals can be made Hurewicz by a forcing
not adding dominating reals.

 A subset $Y$ of $\w^\w$ is said to be \emph{finitely dominating} if
the set $\{\mathrm{maxfin}(F)\colon F\in [Y]^{<\w}\}$ is dominating, where
$\mathrm{maxfin}(F)$ is the coordinatewise maximum of $F$.
 It was shown in \cite{Tsa01-02} that a
space $X\subset\w^\w$ has the Scheepers property if and only if any
continuous image $Y\subset\w^\w$ of $X$ is not finitely dominating.
It is clear that if a finitely dominating set $Y$ becomes bounded in
$V^\IP$ for some poset $\IP$, then $\IP$ adds a dominating real. It
is a classical result of Hurewicz \cite[Theorem~4.3]{COC2} that
$X\subset\w^\w$ is Hurewicz if and only if all its continuous images
$Y\subset\w^\w$ are bounded. Thus if a non-Scheepers subspace $X$ of
$\w^\w$ becomes Hurewicz in $V^\IP$ for certain $\IP$, then $\IP$
adds a dominating real. We do not know whether the converse
implication is true.

\begin{question} \label{q1}
Let $X\subset\w^\w$ be a Scheepers space. Is there a (c.c.c.) poset
$\IP$ which does not add dominating reals and is such that $X$ is
Hurewicz in $V^{\IP}$?
\end{question}

The following result may be
thought of as a  step towards answering  Question~\ref{q1}.

\begin{theorem} \label{turn_into_hur}
Let $\G$ be a  Menger filter. Then there exists a c.c.c.\,poset
$\IP_\G$ which does not add dominating reals and is such that the
filter $\G'=\bigcup_{G\in\G}\uparrow G$ generated by $\G$ in
$V^{\IP_\G}$ is Hurewicz in $V^{\IP_\G}$.
\end{theorem}
\begin{proof}
We shall divide the proof into a sequence of auxiliary statements.
\begin{lemma} \label{prod_imp}
Let $n\in\w$ and  $\F_i$ be a filter for all  $i\in n$. If
$\prod_{i\in n}\F_i$ is Menger, then $\IP=\prod_{i\in n}\mathbb
M_{\F_i}$ does not add a dominating real.
\end{lemma}
\begin{proof}
The proof will be similar to that of  the ``if'' part  of
Theorem~\ref{main2}. However, we shall present it for the sake of completeness.

 Suppose, to the contrary, that $\name{g}$ is a
$\IP$-name for a dominating function. An element of $\IP$ may be
naturally identified with a sequence $\left\la \vec{s},\vec{F}\right\ra$, where
$\vec{s}=\left\la s(i)\colon i\in n\right\ra\in \left([\w]^{<\w}\right)^n$ and
$\vec{F}=\la F(i)\colon i\in n\ra\in \prod_{i\in n}\F_i$, and $\la s(i),
F(i)\ra\in\mathbb M_{\F_i}$ for all $i$.

For every $f\in\w^\w$ let us find  $n_f\in\w$ and a condition
$\left\la\vec{s}^{\,f},\vec{F}^f\right\ra$  forcing
$f(n)<\name{g}(n)$ for all $n\geq n_f$. Since $\w^\w$ cannot be
covered by a countable family of non-dominating sets, we may assume
that $\vec{s}^{\,f}$ and $n_f$ do not depend on $f$, i.e.,
$\vec{s}^{\,f}=\vec{s}_*$ and $n_f=n_*$ for all $f\in\w^\w$.

For every $m\in\w$ let $\mathcal S_m$ be the set of those
$\vec{s}=\la s(i)\colon i\in n\ra\in \left([\w]^{<\w}\right)^n$ such that $\max
s_*(i)<\min s(i)$ for all $i$, and there exists $\vec{F}_{\vec{s}}$
such that $\big\la\la s_*(i)\cup s(i)\colon i\in n\ra,
\vec{F}_{\vec{s}}\big\ra$ forces  $\name{g}(m)$ to be equal to some
$g_{\vec{s}}(m)$. It is clear that for every $\vec{F}\in\prod_{i\in
n}\F_i$ there exists $\vec{s}\in\mathcal S_m$ such that $s(i)\subset
F(i)$ for all $i\in n$. In other words,
\[ \U_m:=\left\{\prod_{i\in n}\uparrow s(i)\colon \la s(i)\colon i\in n\ra\in \mathcal S_m\right\} \]
is an open cover of $\prod_{i\in n}\F_i$. Since the latter product
is Menger, for every $m$ there exists a finite $\V_m\subset\U_m$
such that $\prod_{i\in n}\F_i\subset\bigcup_{m\geq l}\bigcup\V_m$ for
all $l\in\w$. Let $\mathcal T_m\in \left[\mathcal S_m\right]^{<\w}$ be such
that $\V_m=\left\{\prod_{i\in n}\uparrow s(i)\colon \la s(i)\colon i\in n\ra\in
\mathcal T_m\right\}$ and $f(m)=\max\left\{g_{\vec{s}}(m)\colon \vec{s}\in\mathcal
T_m\right\}$. Let also $m\geq n_*$ be such that $\vec{F}^f\in\prod_{i\in
n}\uparrow s(i)$ for some $\la s(i)\colon i\in n\ra\in \mathcal T_m$. It
follows that $s(i)\subset F^f(i)$ for all $i\in n$, and hence $$
p:=\big\la \la s_*(i)\cup s(i)\colon i\in n\ra, \la F^f(i)\cap
F_{\vec{s}}(i)\colon i\in n\ra  \big\ra $$ is a condition in $\IP$
stronger than both $\left\la\vec{s}_*, \vec{F}^f\right\ra$ and $\left\la\la
s_*(i)\cup s(i)\colon i\in n\ra, \vec{F}_{\vec{s}}\right\ra$. However,
$\left\la\la s_*(i)\cup s(i)\colon i\in n\ra, \vec{F}_{\vec{s}}\right\ra$ forces
$\name{g}(m)=g_{\vec{s}}(m)\leq f(m)$, whereas $\left\la\vec{s}_*,
\vec{F}^f\right\ra$ forces $\name{g}(m)>f(m)$, a contradiction.
\end{proof}

\begin{corollary}\label{cor_aux1}
Let $\mathscr F$ be a collection of filters on $\w$.  If $\prod_{i\in
n}\F_i$ is Menger for any $n\in\w$ and $\la \F_i\colon i\in n\ra\in\mathscr
F^n$, then $\IP=\prod_{\F\in \mathscr F}\mathbb M_{\F}$ with finite
supports  does not add a dominating real.
\end{corollary}
\begin{proof}
Since $\IP$ is c.c.c.\,\cite[Theorem 15.15]{Jech03}, if $\IP$ added a
dominating real then there would exist a countable $\mathscr F'\in
[\mathscr F]^\w$ such that $\IP'=\prod_{\F\in \mathscr F'}\mathbb
M_{\F}$ adds a dominating real. The latter product may be viewed as
a finite support iteration whose initial segments are equivalent to
finite products of Mathias posets with respect to elements of
$\mathscr F'$, and hence these initial segments preserve $V\cap\w^\w$
unbounded by our assumption. But then by  \cite[Lemma
6.5.7]{BarJud95} we have that $V\cap\w^\w$ is unbounded in
$V^{\IP'}$ as well, a contradiction.
\end{proof}

In general it is a notorious open question whether it is consistent
that  the Menger property is preserved by finite products. The
following simple statement gives the  answer in the case of  filters.

\begin{claim}\label{filter_powers}
Let $\F$ be a Menger (Hurewicz) filter. Then all finite powers of
$\F$ are Menger (Hurewicz).
\end{claim}
\begin{proof}
Let us fix $n\in\w$ and consider the  map $\phi\colon \F\times \mathcal
P(\w)^n\to\mathcal P(\w)^n$ assigning to $\la F;A_0,\ldots,
A_{n-1}\ra$ the sequence $\la F\cup A_0,\ldots, F\cup A_{n-1}\ra$.
It is clear that the range of $\phi$ is $\F^n$. Since the Menger
(Hurewicz) property is preserved by products with compact spaces and
continuous images, we conclude that $\F^n$ is Menger (Hurewicz).
\end{proof}

Set $\IP_\G=\prod_{\alpha\in\w_1}\mathbb M_{\G}$ with finite
supports and let $\vec{X}=\left\la X_\alpha\colon \alpha<\w_1 \right\ra$ be the
sequence of generic reals added by $\IP_\G$. By
Claim~\ref{filter_powers} and Corollary~\ref{cor_aux1} we have that
$\IP_\G$ does not add dominating reals.
  Let
$R=\bigcap_{n\in\w} O_n\in V[\vec{X}]$ be a $G_\delta$ subset
containing $\G'$. By Claim~\ref{clear_cover_1} we may assume that
$O_n=\bigcup_{a\in A_n}\uparrow a$ for some $A_n\in \left(\G^{<\w}\right)^+$.
Let $\alpha\in\w_1$ be such that $\la A_n\colon n\in\w\ra\in V[\la
X_\xi\colon \xi<\alpha\ra]$. Since $X_\alpha$ is generic over $V[\la
X_\xi\colon \xi<\alpha\ra]$, for every $n\in\w$ there exist infinitely
many $a\in A_n$ such that $a\subset X_\alpha$. In other words,
$\bigcup_{n\in\w}\uparrow (X_\alpha\setminus n)\subset R$. On the
other hand, $\G'\subset \bigcup_{n\in\w}\uparrow (X_\alpha\setminus
n)$. Thus we have found a $\sigma$-compact set containing $\G'$ and
contained in $R$. By \cite[Theorem~5.7]{COC2} this completes our
proof.
\end{proof}

We do not know whether $\G$ itself becomes Hurewicz in the forcing
extension by~$\IP_\G$.
\medskip

\noindent\textbf{Remark 5.6.} Let $\mathscr F$ be a family of filters satisfying the
premises of Corollary~\ref{cor_aux1}. The proof of Theorem~\ref{turn_into_hur}
actually allows us fo find a poset $\IP$ which does not add dominating reals and such that
$\uparrow \F$ is Hurewicz in $V^{\IP}$ for all $\F\in\mathscr F$.

\end{document}